\documentclass[11pt]{article}
\usepackage{amssymb}
\usepackage{mathrsfs}
\usepackage{latexsym,amsmath,amssymb,amsfonts,epsfig,graphicx,cite,psfrag}
\usepackage{eepic,color,colordvi,amscd}
\usepackage{ebezier}
\usepackage{verbatim}
\usepackage{subfigure}

\usepackage{amsthm}
\usepackage{comment}

\theoremstyle{plain}
\newtheorem{theo}{Theorem}[section]

\newtheorem{lem}[theo]{Lemma}

\theoremstyle{definition}

\theoremstyle{remark}

\begin{document}
\title{Wheels in planar graphs and Haj\'{o}s graphs  \footnote{Partially supported by NSF grant DMS-1600738.
Email: qqxie@fudan.edu.cn (Q. Xie), shijie.xie@gatech.edu (S. Xie), yu@math.gatech.edu (X. Yu), and  xyuan@gatech.edu (X. Yuan)}}

\author{Qiqin Xie\\ Shanghai Center of Mathematics\\ Fudan University\\Shanghai, China, 200438\\
\medskip\\
Shijie Xie, Xingxing Yu, and Xiaofan Yuan\\
School of Mathematics\\
Georgia Institute of Technology\\
Atlanta, GA 30332--0160, USA }

\date{November 5, 2019}

\maketitle

\begin{abstract}

It was conjectured by
Haj\'{o}s that graphs containing no $K_5$-subdivision are 4-colorable. Previous results show that any possible minimum counterexample to Haj\'{o}s'
conjecture, called Haj\'{o}s graph, is 4-connected but not 5-connected. In this paper, we show that if a Haj\'{o}s graph
admits a 4-cut or 5-cut with a planar side  then the planar side must be small  or contains a special wheel.
This is a step in our effort to reduce Haj\'{o}s' conjecture to the Four Color Theorem.

\bigskip


\end{abstract}


\section{Introduction}

A  \textit{wheel} is a graph which consists of a cycle,  a vertex not on the cycle (known as the {\it center} of the wheel),
and at least three edges from the center to the cycle. Wheels have played important roles in studing graph structures, e.g.,
Tutte's characterization of 3-connected graphs \cite{Tu61}. Recently,
Aboulker, Chudnovsky, Seymour, and Trotignon \cite{ACST15} characterized 3-connected planar graphs that do not contain a wheel as an induced
subgraph, and used it to show that planar graphs contain no induced wheel are 3-colorable.
In \cite{MY10}, wheels are used to prove that certain 5-connected graphs
contain  a subdivision of $K_5$.

We are interested in the question on when a planar graph contains a wheel that can be extended by disjoint paths to a given set of vertitces. This was motivated by Haj\'{o}s' conjecture that graphs
containing no subdivisions of $K_5$ are 4-colorable -- one of the two remaining open cases of a more general conjecture made by
Haj\'{o}s in the 1950s (see \cite{Th05}, although reference
\cite{Ha61} is often cited.)
Note that if $W$ is a wheel with center $w$ and $w_1,w_2,w_3,w_4$ are distinct
neighbors of $w$ that  occur on $W-w$ in cyclic order, then we can form a subdivision of $K_5$ by adding disjoint paths from $w_1,w_2$ to $w_3,w_4$, respectively, that are also internally
disjoint from $W$.

We say that a graph $G$ is a {\it Haj\'{o}s graph} if
\begin{itemize}
	\item [(1)] $G$ contains no $K_5$-subdivision,
	\item [(2)] $G$ is not 4-colorable, and
	\item [(3)] subject to (1) and (2), $|V(G)|$ is minimum.
\end{itemize}
Thus, if no Haj\'{o}s graph exists then graphs containing no
$K_5$-subdivisions are 4-colorable.
Haj\'{o}s graphs must be 4-connected  \cite{Yu06} but not 5-connected \cite{KSC1, KSC2, KSC3, KSC4}. 
Our goal is to derive more information about Haj\'{o}s graphs in an attempt to reduce  Haj\'{o}s' conjecture to the Four Color Theorem.

To state our result precisely, we need some notation. Let $G_1$, $G_2$
be two graphs.
We use $G_1 \cup G_2$ (respectively, $G_1\cap G_2$) to denote the graph with vertex set
$V(G_1) \cup V(G_2)$ (respectively, $V(G_1)\cap V(G_2)$) and edge set
$E(G_1) \cup E(G_2)$ (respectively, $E(G_1)\cap E(G_2)$).   Let $G$ be a
graph and $k$ a nonnegative integer; then a \textit{k-separation} in $G$ is a pair $(G_1,G_2)$ of edge-disjoint
subgraphs $G_1,G_2$ of $G$ such that $G = G_1 \cup G_2$,  $|V(G_1\cap G_2)|=k$, and $E(G_i)\cup
(V(G_i)\setminus V(G_{3-i}))\ne \emptyset$ for $i=1,2$.

For a wheel $W$ with center $w$ in a graph $G$ and for any $S\subseteq V(G-w)$, we say that $W$ is {\it $S$-good} if $S\cap V(W)\subseteq N_G(w)$, where
$N_G(w)$ denotes the set of neighborhood of $w$ in $G$.
 In Figure~\ref{obstructions}, we list six graphs drawn in a closed disc with $S$ consisting of five vertices on the boundary
of the disc.
Note that none of these graphs contains  an $S$-good wheel.
\begin{figure}[h]
        \includegraphics[width=\textwidth]{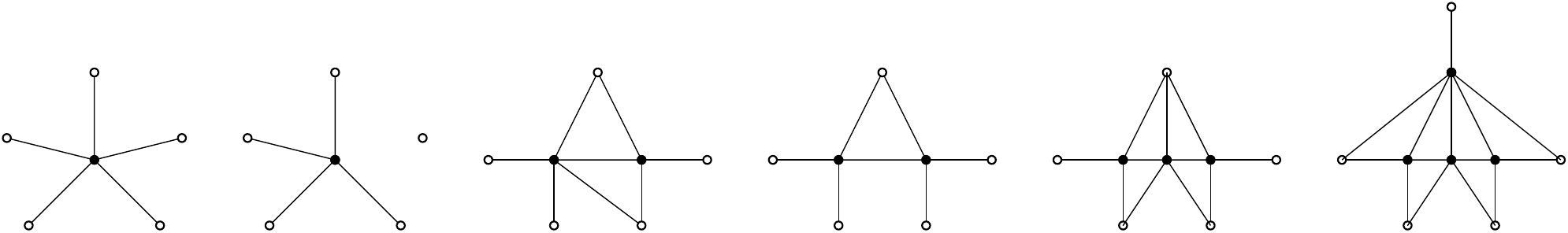}
           \caption{Obstructions to good wheels.}
             \label{obstructions}
  \centering
\end{figure}

Given a graph $G$ and $S\subseteq V(G)$, we say that $(G,S)$ is {\it planar} if it can be drawn in a closed disc in the plane
with no edge crossings and $S$ is contained in the boundary of the disc.
Our result can be stated as follows, it will be used in a subsequent paper to show that
Haj\'os graph has no 4-separation $(G_1, G_2)$ such that $(G_1, V(G_1\cap G_2))$ is planar and $|V(G_1)|\ge 6$, a step in an attempt to reduce the
Haj\'{o}s conjecture to the Four Color Theorem. However, we need to
consider 5-separations as well when we try to extend a wheel to a subdivision of $K_5$.

\begin{theo}\label{main}
Let $G$ be a Haj\'os graph and  $(G_1, G_2)$ be a separation in $G$ of order 4 or 5 such that $(G_1,V(G_1\cap G_2))$ is planar and 
$V(G_1)\setminus V(G_2)\ne \emptyset$.
Then  one of the following holds:
\begin{itemize}
	\item $G_1$ contains a $V(G_1\cap G_2)$-good wheel.
	\item $|V(G_1\cap G_2)|=4$ and $|V(G_1)|=5$.
	\item $|V(G_1\cap G_2)|=5$, $G_1$ is one of the graphs in Figure~\ref{obstructions}, and if $|V(G_1)|=8$ then the degree 3 vertex in $G_1$  has degree at least 5 in $G$.
\end{itemize}
\end{theo}

In Section 2, we deal with several cases when $G_1$ is small. In Section 3, we deal with 4-separations, and in Section 4, we deal with 5-separations.

It will be convenient to use a sequence of vertices to represent a path, with consecutive vertices representing an edge in the path.
Let $G$ be a graph. For $S\subseteq V(G)$, and for any set $T$ of $2$-element subsets of $V(G)$, 
we use $G-(S\cup T)$ to denote the subgraph of $G$ with $V(G-(S\cup T))=V(G) \setminus  S$ and $E(G-(S\cup T))=E(G[V(G)\setminus S])\setminus T$, where $G[V(G)\setminus S]$ denotes the subgraph of $G$ induced by vertices in $V(G)\setminus S$,
and write $G-x$ when $S=\{x\}$ and $T=\emptyset$. 
For any set $S$ disjoint from $V(G)$ and any set $T$ of 2-element subsets of $V(G)\cup S$, we use
$G+(S\cup T)$ to denote the graph with $V(G+(S\cup T))=V(G)$ and $E(G+(S\cup T))=E(G)\cup T$, and write $G+xy$ if $S=\emptyset$ and $T=\{\{x,y\}\}$ with $x,y\in V(G)$.
Let $C$ be a cycle in a plane graph, and let  $u,v\in V(C)$. If $u=v$  let $uCv=u$, and if $u\ne v$ let
$uCv$ denote the subpath of $C$ from $u$ to $v$ in clockwise order.


\section{Small graphs}

In this section, we consider situations when a Haj\'{o}s graph has a separation of order at most 5 and one side of the separation
is a special small graph.  We need the following result from \cite{Yu06}.

\begin{lem} [Yu and Zickfeld] \label{4conn}
Haj\'{o}s graphs are 4-connected.
\end{lem}

We first deal with  4-separations in a Haj\'{o}s graph with one side having six vertices.

\begin{lem}\label{4cut-k2}
 Let $G$ be a Haj\'os graph and let $(G_1,G_2)$ be a $4$-separation in
    $G$ such that $(G_1,V(G_1\cap G_2))$ is planar. Then $|V(G_1)|\ne 6$.
\end{lem}

\begin{proof}
For, suppose $|V(G_1)|=6$. Let $V(G_1)\setminus V(G_1\cap G_2)=\{u,v\}$ and $V(G_1\cap G_2)=\{v_i : i\in [4]\}$, and
assume that $G_1$ is drawn in a closed disc in the plane with no edge crossings such that $v_1, v_2, v_3, v_4$ occur on the boundary of a disc
in clockwise order.
 Then, since
$G$ is 4-connected (by Lemma~\ref{4conn}), we may further assume
that $N_G(u)=\{v_1,v_2,v_3,v\}$ and $N_G(v)=\{v_1,v_3,v_4,u\}$.

Now $G':=G-\{u,v\}+v_2v_4$ contains no $K_5$-subdivision. For, if $T'$ is a
$K_5$-subdivision in $G'$, then
$(T'-v_2v_4)\cup v_2uvv_4$ (and, hence, $G$) contains a $K_5$-subdivision,
a contradiction.

Thus $G'$ has a  4-coloring, say  $\sigma$. If $\sigma(v_2)\in
\{\sigma(v_1),\sigma(v_3)\}$ then we can extend $\sigma$ to a
4-coloring of $G$ by greedily coloring $v,u$ in order. If $\sigma(v_2)\notin
\{\sigma(v_1),\sigma(v_3)\}$ then we can extend $\sigma$ to a
4-coloring of $G$ by coloring $v$ with $\sigma(v_2)$ and coloring $u$
greedily. Either way, we obtain a contradiction to the assumption that $G$ is a Haj\'os graph.
\end{proof}

Next two lemmas deal with 5-separations in Haj\'{o}s graphs with 8 vertices on one side.

\begin{lem}\label{5cuttri}
    Let $G$ be a Haj\'os graph and let $(G_1,G_2)$ be a 5-separation in
    $G$ such that $(G_1,V(G_1\cap G_2))$ is planar.
    Then $G_1-V(G_1\cap G_2) \not \cong K_3$.
\end{lem}

\begin{proof}
     Suppose for a contradiction that $G_1-V(G_1\cap G_2) \cong K_3$. Let
     $u,v,w\in V(G_1)\setminus V(G_2)$ and $V(G_1\cap
     G_2)=\{v_i:i\in [ 5]\}$, and assume that $G_1$ is drawn in a closed disc in the plane with no edge crossings,
     such that $ v_1, v_2, v_3, v_4, v_5$ occur on the boundary of
the disc in clockwise order.

 	Note that $N_G(v_i)\cap \{u,v,w\}\ne \emptyset$ for $i\in [5]$. For, otherwise, we may assume by symmetry that
   $N_G(v_5)\cap \{u,v,w\}=\emptyset$. Since $G$ is 4-connected (by Lemma~\ref{4conn}), $N_G(v_i)\cap \{u,v,w\}\ne 
   \emptyset$ for $i\in [4]$. Since $(G_1,V(G_1\cap G_2))$ is planar, there exists $k\in [4]$ such that $|N_G(v_k)\cap 
   \{u,v,w\}|=1$. Hence, 
   $G$ has a 4-separation
    $(H_1,H_2)$ such that $V(H_1\cap H_2)=\{v_i: i\in [4]\setminus \{k\}\}\cup (N_G(v_k)\cap \{u,v,w\})$ and $H_1=G_1-\{v_k,v_5\}$.
   Now $|V(H_1)|=6$ and $(H_1, V(H_1\cap H_2))$ is planar (as $(G_1,V(G_1\cap G_2))$ is planar), contradicting
    Lemma~\ref{4cut-k2}.

   Moreover, no vertex in $\{u,v,w\}$ is adjacent to four vertices in
   $V(G_1\cap G_2)$. For, suppose $vv_i\in E(G)$ for $i\in [4]$. Then, by planarity and connectivity, $G$ has a separation
    $(H_1,H_2)$ such that $V(H_1\cap H_2)=\{v, v_1,v_4,v_5\}$,
    $(H_1, V(H_1\cap H_2))$ is planar, and
    $V(H_1)=\{v_1,v_4,v_5,u,v,w\}$. We have a contradiction to
    Lemma~\ref{4cut-k2}.

    Also note that any two vertices of $\{u,v,w\}$ must have at least four
    neighbors in $V(G_1\cap G_2)$. For, suppose $u,v$ has at most
    three neighbors in $V(G_1\cap G_2)$. Then $|(N_G(u)\cup N_G(v))\setminus \{u,v\}|=4$ (as $G$ is 4-connected), and $G$ has a separation
    $(H_1,H_2)$ such that $V(H_1\cap H_2)=(N_G(u)\cup N_G(v))\setminus \{u,v\}$, $(H_1, V(H_1\cap H_2))$ is planar, and $V(H_1)=
N_G(u)\cup N_G(v)$ (so
    $|V(H_1)|=6$). Again, we have a contradiction to Lemma~\ref{4cut-k2}.

\medskip

    {\it Case 1}. There exists $\{a,b\}\subseteq \{u,v,w\}$ such that $V(G_1\cap G_2)\subseteq N_G(\{a,b\})$.

     Without loss of generality, we may assume that $a=v$ and $b=w$, and that $v_1,v_2\in N_G(v)$ and $v_3,v_4,v_5\in N_G(w)$. We may further assume
     that the notation is chosen so that $uv_1,uv_5\in E(G)$ (by planarity). Moreover, $vv_3\in E(G)$ since  $u$ and $v$ together must have at least four
    neighbors in $V(G_1\cap G_2)$.

    Let  $G':=G-\{u,v,w\}+\{v_5v_1,v_5v_2,v_5v_3\}$. We claim that
    $G'$ contains no $K_5$-subdivision. For, suppose $T$ is a
    $K_5$-subdivision in $G'$.  If $v_5v_1,v_5v_2,v_5v_3\in E(T)$ then
    $T-\{v_5v_1,v_5v_2,v_5v_3\}$ and  the paths $wv_5,wuv_1,wvv_2,wv_3$ form
    a $K_5$-subdivision in $G$, a contradiction. So $\{v_5v_1,v_5v_2,v_5v_3\}\not\subseteq E(T)$. Then we obtain a contradiction by forming a  $K_5$-subdivision in $G$ from $T$: replacing edges in $\{v_5v_1,v_5v_2,v_5v_3\}\cap E(T)$ with one or two paths from $\{v_5uv_1,v_5wvv_2\}$, or
    $\{v_5uv_1,v_5wv_3\}$, or $\{v_5uvv_2,v_5wv_3\}$.

   Thus,  $G'$ has a 4-coloring, say $\sigma$. We have a contradiction by extending
   $\sigma$ to a 4-coloring of $G$: first assign
   $\sigma(v_5)$ to $v$, and then greedily color $w,u$ in order.

\medskip

   {\it Case} 2. For any $\{a,b\}\subseteq \{u,v,w\}$, $|N_G(\{a,b\})\cap V(G_1\cap G_2)|=4$.

 Without loss of generality, we may assume that $uv_1,uv_5,vv_2,vv_3\in
 E(G)$. By symmetry and planarity, assume $wv_3,wv_4\in
 E(G)$.  Note that $\{wv_5,vv_1\}\not\subseteq E(G)$ as, otherwise, $V(G_1\cap G_2)\subseteq N_G(\{v,w\})$ (but we
 are in Case 2). On the other hand,  since any two vertices of $\{u,v,w\}$ must have at least four
    neighbors in $V(G_1\cap G_2)$, $wv_5\in E(G)$ or $vv_1\in E(G)$.
So by symmetry,  we may assume $wv_5\in E(G)$
and $vv_1\notin E(G)$.

Let  $G':=G-\{u,v,w\}+\{v_5v_2,v_5v_3\}$. Note that $G'$ contain
no $K_5$-subdivision as $v_5v_2,v_5v_3$ can be replaced by $v_5uvv_2,v_5wv_3$, respectively.
Hence $G'$ has a  4-coloring, say $\sigma$. By assigning $\sigma(v_5)$ to $v$ and
greedily coloring $w,u$ in order, we obtain a 4-coloring of $G$, a contradiction.
 \end{proof}

\medskip

Now we characterize the situation when the planar side of a 5-separation in a Haj\'os graph has exactly eight vertices.

    \begin{lem}\label{5cut3vts}
      Let $G$ be a Haj\'os graph and let $(G_1,G_2)$ be a 5-separation in
     $G$ such that $|V(G_1)|= 8$, $(G_1,V(G_1\cap G_2))$ is planar, and
     $V(G_1\cap G_2)$ is an independent set in $G_1$. Then $(G_1,V(G_1\cap G_2))$
     is the 8-vertex graph in Figure~\ref{obstructions}, where the vertex in $V(G_1\cap G_2)$
     with three neighbors in $V(G_1)\setminus V(G_1\cap G_2)$ has degree at least 5 in $G$.
      \end{lem}

\begin{proof}
        Let $V(G_1\cap G_2)=\{t_i: i\in [5]\}$ and we may assume
        that $G_1$ is drawn in a closed disc in the plane such that $t_1,t_2,t_3,t_4, t_5$ occur on the boundary of the disc in
        clockwise order. Let $D:=G_1-V(G_1\cap G_2)$.
       Since $G$ is 4-connected (by Lemma~\ref{4conn}) and $(G_1,V(G_1\cap G_2))$ is planar,
       $D$ must be connected. Thus, by Lemma~\ref{5cuttri}, $D$ is a
       path of length two, and we write $D=uvw$.

       Suppose that $|N_G(u)\cap N_G(w)|\ge 3$. Then, since $G$ is
       4-connected and $(G_1,V(G_1\cap G_2))$ is planar,
       we may assume that $t_1,t_4\in N_G(u)\cap N_G(w)$  and
       $N_G(w)=\{t_1,t_4,t_5,v\}$. Hence, by planarity, $G$ has a
       separation $(H_1,H_2)$ such that $V(H_1\cap
       H_2)=\{t_1,t_4,t_5,u\}$, $(H_1,V(H_1\cap H_2))$ is planar, and $H_1=G_1-\{t_2,t_3\}$,
       contradicting Lemma~\ref{4cut-k2}.

      Thus, since $G$ is 4-connected and $(G_1,V(G_1\cap G_2))$ is
      planar, we may assume that $N_G(u)=\{t_1,t_2,t_3,v\}$ and
      $N_G(w)=\{t_4,t_5,t_1,v\}$.

     Then $vt_3,vt_4\in
      E(G)$. For, otherwise, we may assume by symmetry that
      $vt_4\notin E(G)$. Now $G$ has a separation $(H_1,H_2)$ such
      that $V(H_1\cap H_2)=\{t_1,t_2,t_3,w\}$, $G_2\subseteq H_2$, $(H_1, V(H_1\cap H_2))$
      is planar, and $H_1=G_1-\{t_4,t_5\}$, contradicting Lemma~\ref{4cut-k2}.

      If $vt_1\notin E(G)$ then let $G'=G-\{u,v,w\}+\{t_1t_3, t_1t_4\}$.
      Note that $G'$ has no $K_5$-subdivision; for,  if $T$ is a
      $K_5$-subdivision in $G'$ then $(T-\{t_1t_3,t_1t_4\})\cup
      t_1ut_3\cup t_1wt_4$ (and hence $G$) also contains a
      $K_5$-subdivision. Thus, since $G$ is a Haj\'{o}s graph, $G'$ is
      4-colorable; and let $\sigma$ be a 4-coloring of $G'$.
       By assigning $\sigma(t_1)$ to $v$ and greedily
      coloring $u,w$ in order, we obtain a 4-coloring of $G$, a
      contradiction.

      So by planarity of $G_1$, $N_G(v)=\{t_1,t_3,t_4,u,w\}$. If
      $d_G(t_1)\ge 5$ then $(G_1,V(G_1\cap G_2))$ is the 8-vertex graph in Figure~\ref{obstructions}, as $V(G_1\cap G_2)$ is independent in $G_1$.
      So assume that
      $d_G(t_1)=4$, and let $G'$ be obtained from $G-\{t_1,v\}$ by
      identifying $u$ and $w$ as $v^*$. Then $G'$ has no
      $K_5$-subdivision; since if
      $T$ is a $K_5$-subdivision in $G'$ then $((T-v^*)+\{v, vt_3,vt_4\})\cup vut_2 \cup vwt_5$ (and hence $G$) also contains  a $K_5$-subdivision. So
      let $\sigma$ be a 4-coloring of $G'$. Then by assigning
      $\sigma(v^*)$ to both $u$ and $w$ and greedily coloring
      $v,t_1$ in order, we extend $\sigma$ to a 4-coloring of
      $G$, a contradiction.
\end{proof}


\section{4-Separations}

In this section we prove that if the planar side of a 4-separation in a Haj\'os graph has at least 6 vertices then it contains a good wheel.

\begin{lem}\label{4cut}
    Let $G$ be a Haj\'os graph and $(G_1,G_2)$ be a 4-separation in
    $G$ such that $|V(G_1)|\ge 6$ and $(G_1,V(G_1\cap G_2))$ is planar.
      Then $G_1$ contains a $V(G_1\cap G_2)$-good wheel.
\end{lem}

\begin{proof}
   We may choose $(G_1,G_2)$ so that $G_1$ is minimal, since,  for any $4$-separation $(G_1',G_2')$
in $G$ with $G_1'\subseteq G_1$, a $V(G_1'\cap G_2')$-good wheel in $G_1'$ is also a $V(G_1\cap G_2)$-good wheel in $G_1$.

By Lemma~\ref{4cut-k2}, $|V(G_1)|\ge 7$.
    Let $V(G_1\cap G_2)=\{t_1, t_2,t_3,t_4\}$ and assume that $G_1$ is drawn in a closed disc in the plane with no edge crossings and with $t_1, t_2, t_3, t_4$ on the boundary of the
disc in clockwise order.    Let $D:=G_1-V(G_1\cap G_2)$. Since $G$ is 4-connected (by Lemma~\ref{4conn}), $|N_G(t_i)\cap V(D)|\ge 1$ for each $i\in [4]$.
    In fact,

   \begin{itemize}
   \item [(1)]  $|N_G(t_i)\cap V(D)|\ge 2$ for $i\in [4]$.
   \end{itemize}
   For, suppose  $|N_G(t_i)\cap V(D)|= 1$ for some $i\in [4]$, and let $t\in N_G(t_i)\cap V(D)$. Then $(G_1-t_i,G_2+\{t,tt_i\})$ is a separation
   in $G$ that contradicts the minimality of $G_1$. $\Box$

   \begin{itemize}
   \item [(2)] $D$ is 2-connected.
   \end{itemize}
   Suppose to the contrary that $D$ is not 2-connected. Then $D$ has a separation $(D_1, D_2)$ such that $|V(D_1\cap D_2)|\le 1$ and  $|V(D_i)\setminus V(D_{3-i})|\ge 1$ for $i=1,2$.
   Since $G$ is 4-connected,
   $|N_G(D_i-D_{3-i})\cap \{t_1,t_2,t_3,t_4\}|\ge 3$ for $i=1,2$.

   Thus by planarity (and choosing appropriate notation for $t_i$), we may assume that $t_1,t_2,t_3\in N_G(D_1-D_2)$ and $t_3,t_4,t_1\in
   N_G(D_2-D_1)$.  Note that $G$ has a separation $(G_1',G_2')$ such that $V(G_1'\cap G_2')=\{t_1,t_2,t_3\}\cup
   V(D_1\cap D_2)$ and $D_1\subseteq G_1'\subseteq G_1$, as well as a separation $(G_1'',G_2'')$ such that $V(G_1''\cap G_2'')=\{t_1,t_3,t_4\}\cup
   V(D_1\cap D_2)$ and $D_2\subseteq G_1''\subseteq G_1$. Since $G$ is 4-connected, $|V(D_1\cap D_2)|=1$.
   Thus by the choice of $(G_1,G_2)$ (the minimality of $G_1$),
  $|V(D_i)|=2$ for $i=1,2$. But then $|N_G(t_2)\cap V(D)|=1$, contradicting (1). $\Box$

   \medskip
   Let $C$ be the outer cycle of $D$. If there exists $x\in
   V(D)\setminus V(C)$ then all vertices and edges  of $D$ cofacial with $x$ (including $x$) form the desired wheel.
  Thus we may assume that

   \begin{itemize}
   \item [(3)] $V(D)=V(C)$.
   \end{itemize}

We claim that

   \begin{itemize}
   \item [(4)] for each $i\in [4]$, there exists $u_i\in N_G(t_i)\cap N_G(t_{i+1})\cap V(C)$ (with $t_1:=t_5$).
   \end{itemize}
   For, suppose (4) fails and, without loss of generality, assume that $N_G(t_4)\cap N_G(t_1)\cap V(C)=\emptyset$. Let $v_1\in N_G(t_1)\cap V(C)$ and
   $v_4\in N_G(t_4)\cap V(C)$ such that $v_4Cv_1$ is minimal. Thus, $v_1t_4,v_4t_1\notin E(G)$.
   By (1) and by planarity, $v_1t_2,v_1t_3,v_4t_2,v_4t_3\notin E(G)$.
Since the degree of $v_1$ in $G$ is at least 4, $v_1$ has a neighbor $v$ in $D$ such that $vv_1\in E(D)\setminus E(C)$. We choose $v$, such that $vCv_1$ is minimal. Moreover, let $v'\in V(v_4Cv_1-v_1)$ such that $v'v_1\in E(C)$.

 Suppose $N_G(t_3)\cap V(v_1Cv-v)=\emptyset$. Then $G$ has a 4-separation $(H_1,H_2)$ such that $V(H_1\cap H_2)=\{t_1,t_2, v,v'\}$,
 $v_1\in V(H_1)\setminus V(H_2)$, $G_2+\{t_3,t_4\}\subseteq H_2$,  $(H_1,V(H_1\cap H_2))$ is planar,
 and $|V(H_1)|\ge 6$ (by (1)). Now $(H_1,H_2)$ contradicts the choice of $(G_1,G_2)$ (the minimality of $G_1$).

  Hence $N_G(t_3)\cap V(v_1Cv-v)\neq\emptyset$.
 Moreover, since $G$ is 4-connected, $\{v_1,v,t_4\}$ cannot be a cut in $G$.
 Hence, $N_G(t_3)\not\subseteq V(v_1Cv)$.
     Thus, $G$ has a 4-separation $(H_1,H_2)$ such that $V(H_1\cap H_2)=\{t_3,t_4,v,v_1\}$, $v_4\in V(H_1)\setminus V(H_2)$ (so $|V(H_1)|\ge 6$ by (1)),
   $G_2+\{t_1,t_2\}\subseteq H_2$, and $(H_1,V(H_1\cap H_2))$ is planar. Now $(H_1, H_2)$ contradicts  the choice of $(G_1,G_2)$. $\Box$

   \medskip
   We may assume that
  \begin{itemize}
   \item [(5)]  $V(C)=\{u_1,u_2,u_3,u_4\}$.
  \end{itemize}
    First, we may assume that  $N_G(t_i)\cap (V(C)\setminus \{u_j: j\in [4]\})=\emptyset$
     for $i\in [4]$.  For, suppose, without loss of generality, that there exists $u\in
   V(u_4Cu_1)\setminus \{u_4,u_1\}$ with $ut_1\in E(G)$. Then the vertices and edges of $G_1$
   cofacial with $u$ (including $u$) form a $V(G_1\cap
   G_2)$-good wheel.

    Now suppose $V(C)\ne \{u_1,u_2,u_3,u_4\}$.  Then $|V(C)|=5$. For, otherwise, $G$ has a 4-separation $(H_1,H_2)$
 such that $V(H_1\cap H_2)=\{u_1,u_2,u_3,u_4\}$, $H_1=D$, $(H_1,V(H_1\cap H_2))$ is planar,
          and $G_2+\{t_i : i\in [4]\}\subseteq H_2$. Clearly, $(H_1,H_2)$ contradicts the choice of $(G_1,G_2)$.

     So let $u\in V(C)\setminus  \{u_1,u_2,u_3,u_4\}$ and, without loss of generality,
     assume that $u\in V(u_4Cu_1)$.  Since $ut_1\notin E(G)$, $uu_2,uu_3\in E(G)$ as $G$ is 4-connected.
     Hence, $G$ has a 5-separation $(H_1,H_2)$ such that $V(H_1\cap H_2)=\{t_2,t_3,t_4, u_1,u_4\}$
     and $H_1-V(H_1\cap H_2)$ is the triangle $uu_2u_3u$, contradicting Lemma~\ref{5cuttri}.  $\Box$

    \begin{itemize}
   \item [(6)] $D\ne C$.
    \end{itemize}
        For, suppose $D=C$. Let $\sigma$ be a 4-coloring of $G-\{u_1,u_2,u_3,u_4\}$ which exists as $G$ is a Haj\'{o}s graph.
      We can extend $\sigma$ to a 4-coloring of $G$ as follows:
      If $|\{\sigma(t_i): i\in [4]\}|=4$ then assign  to $u_1,u_2,u_3,u_4$ the colors
      $\sigma(t_4),\sigma(t_1), \sigma(t_2), \sigma(t_3)$,
      respectively. If $|\{\sigma(t_i): i\in [4]\}|\le 3$ then assign
      to both $u_1$ and $u_3$ a color not in  $\{\sigma(t_i): i\in [4]\}$, and
      greedily color $u_2,u_4$ in order. This contradicts the assumption that $G$ is a Haj\'{o}s graph. $\Box$

\medskip

    By (6) and by planarity of $D$, we may assume $D=C+u_2u_4$. Note that $G':=(G-\{u_1,u_2,u_3,u_4\})+\{t_1t_2,t_2t_3,t_3t_1\}$ contains no $K_5$-subdivision;
      for, if $T$ is a $K_5$-subdivision in $G'$ then, by replacing $t_1t_2,t_2t_3,t_3t_1$ (whenever in $T$)
      with $t_1u_1t_2, t_2u_2t_3, t_3u_3u_4t_1$, respectively, we obtain a
      $K_5$-subdivision in $G$.

    Thus let $\sigma$ be a 4-coloring of  $G'$.   If  $|\{\sigma(t_i): i\in [4]\}|=4$ then  assign  to
      $u_1,u_2,u_3,u_4$ the colors
      $\sigma(t_4),\sigma(t_1), \sigma(t_2), \sigma(t_3)$,
      respectively, we obtain a 4-coloring of $G$, a
      contradiction.

        So assume $|\{\sigma(t_i): i\in [4]\}|=3$ and
      $\sigma(t_4)\in \{\sigma(t_i): i\in [3]\}$. We derive a
      contradiction by extending $\sigma$ to a  4-coloring of $G$:
      If $\sigma(t_4)=\sigma(t_2)$ or $\sigma(t_4)=\sigma(t_1)$ then assign $\sigma(t_1), \sigma(t_3)$ to $u_2,u_4$, respectively, and greedily color $u_1,u_3$ in order; and
      if $\sigma(t_4)=\sigma(t_3)$ then assign $\sigma(t_1), \sigma(t_2)$ to $u_2,u_4$, respectively, and greedily color $u_1,u_3$ in order.
\end{proof}


\section{5-Separations}

In this section, we characterize the 5-separations $(G_1,G_2)$ in a Haj\'{o}s graph such that  
$(G_1,V(G_1\cap G_2))$ planar and $G_1$ has no
$V(G_1\cap G_2)$-good wheel. First, we deal with the case when every vertex in $V(G_1\cap G_2)$ has at least two 
neighbors in $G_1-G_2$. 

\begin{lem}\label{5cutdeg2}
    Let $G$ be a Haj\'os graph and $(G_1,G_2)$ be a 5-separation
    in $G$ such that $(G_1,V(G_1\cap G_2))$ is
    planar and every vertex in $V(G_1\cap G_2)$ has at least two
    neighbors in $G_1-V(G_1\cap G_2)$. Then $G_1$ contains a $V(G_1\cap G_2)$-good wheel.
\end{lem}

\begin{proof}
     Let $V(G_1\cap G_2)=\{t_i : i\in [5]\}$ and $D=G_1-V(G_1\cap G_2)$,   and
     assume that  $G_1$ is drawn in a closed disc in the plane
     such that $t_1,t_2,t_3,t_4,t_5$ occur on the boundary of
     the disc in clockwise order.  Therefore,  by planarity and the assumption that
     $|N_G(t_i)\cap V(D)|\ge 2$ for $i\in [5]$, we have $|V(D)|\ge 3$.  We claim that

    \begin{itemize}
     \item [(1)] $D$ is 2-connected.
    \end{itemize}
    For, suppose $D$ is not 2-connected.  Then, since $|V(D)|\ge 3$,  $D$ has a separation $(D_1,D_2)$
    such that $|V(D_1\cap D_2)|\le 1$. Since $G$ is 4-connected and $(G_1,V(G_1\cap G_2)$ is planar,  we may assume
    without loss of generality that $t_4,t_5\notin N_G(D_1-D_2)$. Hence, $G$ has a separation $(H_1,H_2)$ such that
    $V(H_1\cap H_2)=\{t_1,t_2,t_3\}\cup V(D_1\cap D_2)$, $D_1\subseteq H_1$, and $G_2\cup D_2\subseteq H_2$.
     Note $|V(D_1\cap D_2)|=1$ as $G$ is 4-connected.
       If $|V(D_1)|=2$ then $|N_G(t_2)\cap V(D)|=1$, a contradiction. So  $|V(D_1)|\ge 3$ and, hence, $|V(H_1)|\ge 6$.  By Lemma~\ref{4cut-k2}, 
     $H_1$ has a $V(H_1\cap H_2)$-good wheel, which is also a $V(G_1\cap G_2)$-good wheel in $G_1$, a contradiction. $\Box$

  \medskip

   By (1), let $C$ be the outer cycle of $D$ and, for $i\in [5]$, let
   $v_i,w_i\in N_{G_1}(t_i)\cap V(C)$ such that $w_iCv_i$ contains $N_{G_1}(t_j)\cap V(C)$
   for $j\in [5]\setminus \{i\}$ and, subject to this, $w_iCv_i$ is minimal.
   Then $v_1,w_1,v_2,w_2, v_3,w_3,v_4,w_4,v_5,w_5$ occur on $C$ in
   clockwise order.  We may assume

    \begin{itemize}
     \item [(2)] $V(D)= V(C)$ and $N_{G_1}(t_i)\cap V(C)=\{v_i,w_i\}$ for
       $i\in [5]$.
    \end{itemize}
    For, suppose there exists $v\in V(D)\setminus
    V(C)$ or $v\in N_{G_1}(t_i)\cap V(v_iCw_i-\{v_i,w_i\})$ for some $i\in [5]$. Then, since $G$ is 4-connected, all vertices and edges in $G_1$
    cofacial with $v$ (including $v$) form a $V(G_1\cap
    G_2)$-good wheel.  $\Box$
    
\medskip

  We may also assume that

  \begin{itemize}
     \item [(3)] $|\{i\in [5]: w_i\ne v_{i+1}\}|\le 1$, and if $w_i\ne v_{i+1}$ then $w_i$ and $v_{i+1}$
       each have a neighbor in $v_{i+3}Cw_{i+3}-\{v_{i+3},w_{i+3}\}$, where $v_k=v_{k-5}$ and $w_k=w_{k-5}$
       for $k>5$.
  \end{itemize}
       Note that if the second half of (3) holds then the first half of (3) follows from the planarity of $G_1$.
       Thus, we only consider the second half of
      (3) with $w_5\ne v_1$, without lost of generality.

         Then $w_5t_1,v_1t_5\notin E(G)$. By planarity of $G_1$,
       neither $v_1$ nor $w_5$ is adjacent to any of $\{t_2,t_3,t_4\}$.
       Hence, since $G$ is 4-connected,
       there exist $w_5w_5',v_1v_1'\in E(G)\setminus E(C)$, where
       $v_1',w_5'\in V(C)$ (by (2)). By planarity of $G_1$ again, $v_1,v_1',w_5',w_5$ occur on
       $C$ in this clockwise order. Choose these edges so that
       $v_1'Cw_5'$ is minimal.

     It suffices to show   that $v_1'Cw_5'\subseteq v_3Cw_3-\{v_3,w_3\}$.
     For, suppose $v_1'Cw_5'\not \subseteq v_3Cw_3-\{v_3,w_3\}$.
      Then we may assume by symmetry that $w_5'\in V(w_3Cw_5-w_5)$. Since $G$
      is 4-connected, $\{t_5,w_5,w_5'\}$ cannot be a cut in $G$; so
      $w_5'\in V(w_3Cw_4-w_4)$. Let $w$ be the neighbor of $w_5$ with $w_5w\in 
      E(w_5Cv_1)$. Then $G_1$ has a 4-separation $(G_1',G_1'')$ such that
      $V(G_1'\cap G_1'')=\{t_4,t_5,w,w_5'\}$, $V(G_1\cap G_2)\subseteq
      V(G_1'')$, and $|V(G_1')|\ge 6$. By Lemma~\ref{4cut}, $G_1'$ contains a
      $V(G_1'\cap G_1'')$-good wheel, which is also a $V(G_1\cap G_2)$-good wheel in $G_1$.       $\Box$

\medskip

      Thus, by (3), we have two cases.

    \medskip
     {\it Case} 1.  $|\{i\in [5]: w_i\ne v_{i+1}\}|=1$.

      Without loss of generality, assume that $w_5\ne v_1$, and, by
      (3), let $v_1',w_5'\in V(v_3Cw_3)\setminus \{v_3,w_3\}$ such that $w_5w_5',v_1v_1'\in
      E(G)$  and, subject to this,
             $v_1'Cw_5'$ is minimal.

     We may further assume that $|V(w_5'Cw_5)|=4$. To see this, we first
     note that, by (2), $G_1$ has a 4-separation $(G_1',G_1'')$ such that $V(G_1'\cap
     G_1'')=\{w_5,w_5',v_4,w_4\}$ and $V(G_1\cap G_2)\subseteq V(G_1'')$. If $|V(G_1')|\ge 6$ then,
     by Lemma~\ref{4cut}, $G_1'$ contains a
      $V(G_1'\cap G_1'')$-good wheel, which is also a $V(G_1\cap G_2)$-good wheel in $G_1$. Thus, we may
     assume $|V(G_1')|\le 5$. If $|V(G_1')|=4$ then $|V(w_5'Cw_5)|=4$. So assume that there exists
     $u\in V(w_5'Cw_5)\setminus \{v_4,w_4, w_5,w_5'\}$.  If $u\in V(w_5'Cv_4)$
     then $uw_4,uw_5\in E(G)$ and
     $G$ has a 5-separation $(H_1,H_2)$ such that $V(H_1\cap H_2)
     =\{w_5, w_5',t_3,t_4,t_5\}$ and  $H_1-V(H_1\cap H_2)$ is the triangle $uv_4w_4u$,
     contradicting Lemma~\ref{5cuttri}.
       If $u\in V(v_5Cw_5)$ then $uw_5',uv_4\in E(G)$ and $G$ has a 5-separation $(H_1,H_2)$ such that $V(H_1\cap H_2)
     =\{w_5,w_5',t_3,t_4,t_5\}$ and $H_1-V(H_1\cap H_2)$ is the triangle $uv_4w_4u$, contradicting Lemma~\ref{5cuttri}.
    So $u\in V(v_4Cw_4)$ and, hence, $uw_5,uw_5'\in E(G)$. Let
     $w$ be the neighbor of $w_5$ on $w_5Cv_1$. Then $G$ has a 5-separation $(H_1,H_2)$ such that $V(H_1\cap H_2)
     =\{w,w_5',v_4,t_4,t_5\}$ and $H_1-V(H_1\cap H_2)$ is the triangle $uw_4w_5u$, again
     contradicting Lemma~\ref{5cuttri}.

     By symmetry, we may also assume that $|V(v_1Cv_1')|=4$.

    \medskip
      {\it Subcase} 1.1. $w_5'\ne v_1'$.

     Note that $G_1$ has a 4-separation $(G_1',G_1'')$ such that $V(G_1'\cap
     G_1'')=\{v_1,v_1',w_5,w_5'\}$, $w_5Cv_1\cup v_1'Cw_5'\subseteq G_1'$, and $V(G_1\cap G_2)\subseteq V(G_1'')$.
      If $|V(G_1')|\ge 6$ then by Lemma~\ref{4cut}, $G_1'$ contains a
      $V(G_1'\cap G_1'')$-good wheel, which is also a $V(G_1\cap G_2)$-good wheel in $G_1$.
      Thus, we may assume   $|V(G_1')|\le 5$.

     If $N_G(w_5')\cap V(w_5Cv_1-w_5)=\emptyset$ then, since $G$ is
     4-connected, $w_5'v_5\in E(G)$. Let $w'$ be the neighbor of $w_5'$
     in $v_1'Cw_5'$. Then $G$ has a 5-separation
     $(L_1,L_2)$ such that $V(L_1\cap L_2)=\{w_5,w',t_3,t_4,t_5\}$, $G_2\subseteq L_2$, and $L_1-V(L_1\cap L_2)$
     is the triangle $w_5'v_4v_5w_5'$, contradicting
     Lemma~\ref{5cuttri}.

     Thus, $N_G(w_5')\cap V(w_5Cv_1-w_5)\ne \emptyset$. Since $w_5'\ne
     v_1'$, it follows from the minimality of $v_1'Cw_5'$ that there
     exists $u\in w_5Cv_1-\{w_5,v_1\}$ such that $uw_5' \in
     E(G)$. Hence, since $|V(G_1')|\le 5$ and $G$ is 4-connected,
   $uv_1'\in E(G)$. Now  $G$ has a 5-separation
     $(L_1,L_2)$ such that $V(L_1\cap L_2)=\{v_1,v_1', v_4,v_5,t_5\}$, $G_2\subseteq L_2$, and $L_1-V(L_1\cap L_2)$
     is the triangle $uw_5'w_5u$, contradicting
     Lemma~\ref{5cuttri}.

 \medskip

     {\it Subcase} 1.2. $w_5'=v_1'$.

    Then, since $G$ is 4-connected, $|V(w_5Cv_1)|=2$.
    If $w_5v_4\in E(G)$ then $G$ has a 5-separation
    $(H_1,H_2)$ such that $V(H_1\cap H_2)=\{v_1,v_1',t_3,t_4,t_5\}$
    and $H_1-V(H_1\cap H_2)$ is the triangle $w_5v_4v_5w_5$, contradicting
    Lemma~\ref{5cuttri}. Hence, $w_5v_4\notin
    E(G)$. Similarly, $v_1v_3\notin E(G)$.

   Suppose $w_5'v_5\notin E(G)$ or $v_1'w_1\notin E(G)$. By symmetry,
   we may assume the former. Let
   $G':=G-\{v_4,v_5,w_5\}+\{w_5't_4,w_5't_5\}$. Then $G'$ has no
   $K_5$-subdivision;
    for, if $T$ is a $K_5$-subdivision in $G'$ then
   $(T-\{w_5't_4,w_5't_5\})\cup w_5'v_4t_4\cup w_5'w_5t_5$ (and, hence, $G$) would also contain a
   $K_5$-subdivision. So  let $\sigma$ be a 4-coloring
   of $G'$. By assigning $\sigma(w_5')$ to $v_5$ and
   greedily coloring $v_4,w_5$ in order, we obtain a 4-coloring
   of $G$, a contradiction.

	So $w_5'v_5,v_1'w_1\in E(G)$.  Let $G':=G-D+\{v, t_1t_5, vt_1, vt_2, vt_4, vt_5\}$, where $v$ is a new vertex.
   	We claim that $G'$ has no $K_5$-subdivision. For, suppose $T$ is a $K_5$-subdivision in $G'$, 
	then $(T-\{v, t_1t_5, vt_1, vt_2, vt_4, vt_5\}+{v_1'})\cup t_5w_5v_1t_1\cup v_1'v_2t_1\cup v_1'v_3t_2\cup v_1'v_4t_4 \cup v_1'v_5t_5$ (and, hence, $G$)
	would also contain a $K_5$-subdivision, a contradiction.

   Thus, let $\sigma $ be a 4-coloring of $G'$. We derive a contradiction by extending
   $\sigma$ to a 4-coloring of $G$ as follows: First, assign $\sigma(v)$ to
   both $v_2$ and $v_5$. Then greedily color $v_3,v_4,w_5'$ in
   order. If $w_5'$ receives $\sigma(t_5)$ then greedily color
   $v_1,w_5$ in order. Otherwise, assign $\sigma(t_5)$ to $v_1$ and
   greedily color $w_5$.

    \medskip

      {\it Case} 2.  $|\{i\in [5]: w_i\ne v_{i+1}\}|=0$.

       First, we show that $|V(v_iCw_i)|\le 3$ for $i\in       [5]$.
       For, suppose not and, by symmetry, assume that $|V(v_1Cw_1)|\ge 4$, and
       let $a,b\in V(v_1Cw_1)\setminus \{v_1,w_1\}$ be distinct such that
       $v_1,a,b,w_1$ occur on $C$ in clockwise order and $ab\in E(C)$.
        By planarity, $N_G(b)\cap V(v_4Ca-\{a,v_4\})=\emptyset$ or
        $N_G(a)\cap V(bCv_4-\{b,v_4\})=\emptyset$;
       so by symmetry, we may assume the former.
       Then, since $G$ is 4-connected, there exists $b'\in
       N_G(b)\cap V(v_2Cv_4-v_2)$ and choose $b'$ so that $b'Cv_4$ is minimal.
       Now $G_1$ has a 4-separation $(G_1',G_1'')$ such that $V(G_1'\cap
       G_1'')=\{a,v_2,v_3,b'\}$, $b\in V(G_1')\setminus V(G_1'')$,  and
       $V(G_1\cap G_2)\subseteq V(G_1'')$.  If $|V(G_1')|\ge 6$ then,
     by Lemma~\ref{4cut}, $G_1'$ contains a
      $V(G_1'\cap G_1'')$-good wheel, which is also a $V(G_1\cap G_2)$-good wheel in $G_1$. We may thus assume
       $|V(G_1')|=5$. Hence, $b'\in V(v_3Cv_4-v_3)$ and $bv_3\in
       E(G)$.
       Now $G$ has a 5-separation $(L_1,L_2)$ such that $|V(L_1\cap
       L_2)|=\{a, b',t_1,t_2,t_3\}$ and $L_1-V(L_1\cap L_2)$ is the
       triangle $bv_2v_3b$, contradicting   Lemma~\ref{5cuttri}.

       We may assume  that for $i\in [5]$, if $v_iv_{i+2}\in E(G)$ then
   $v_iCv_{i+2}=v_iv_{i+1}v_{i+2}$. (Here, we let $v_6=v_1$ and $v_7=v_2$.)  For, otherwise, $G_1$ has a
   4-separation $(G_1',G_1'')$ such that $V(G_1'\cap
   G_1'')=\{t_i,t_{i+1},v_i,v_{i+2}\}$, $|V(G_1')|\ge 6$, and
   $V(G_1\cap G_2)\subseteq V(G_1'')$. Hence,
     by Lemma~\ref{4cut}, $G_1'$ contains a
      $V(G_1'\cap G_1'')$-good wheel, which is also a $V(G_1\cap G_2)$-good wheel in $G_1$.

      Next, we claim that $G_2$ has a 4-coloring $\sigma$ with
      $|\sigma(\{t_i: i\in [5]\})|\le 3$. To see this, consider
      $G':=G-D+\{v,vt_i: i\in [5]\}$, where $v$ is a new vertex. If
      $G'$ contains no $K_5$-subdivision then $G'$ is 4-colorable; so
      the desired 4-coloring for $G_2$ exists. Thus, let $T$ be a
      $K_5$-subdivision in $G'$  and assume, without loss of
      generality, that $vt_5\notin E(T)$. Then
      $(T-v)\cup v_3v_2t_1\cup v_3t_2\cup v_3t_3\cup v_3v_4t_4$ and, hence, $G$
      contain a $K_5$-subdivision, a contradiction.

    We assume that the 4-coloring $\sigma$  of
   $G_2$ uses colors from $\{1,2,3,4\}$ and that $\sigma(\{t_i: i\in [5]\}) \subseteq \{1,2,3\}$.
    Note that  $0\le |\{i : i\in [5] \mbox{ and } v_iv_{i+1} \in E(G)\}|\le 5$.

\medskip

   {\it Subcase} 2.1. $|\{i: i\in [5] \mbox{ and } v_iv_{i+1} \in E(G)\}|=0$

    Then by the above assumption, $\{v_i: i\in [5]\}$ is an
    independent set in $G$. Since $D$ is outer planar, we may let $\sigma'$ be a
    3-coloring of $D$ using colors from $\{1,2,3\}$.  By changing
    $\sigma'(v_i)$ to $4$ for $i\in [5]$, we obtain a
    4-coloring $\sigma''$ of $D$ from $\sigma'$. Now $\sigma$ and
    $\sigma''$ form a $4$-coloring of $G$, a contradiction.

\medskip

  {\it Subcase} 2.2. $ |\{i: i\in [5] \mbox{ and } v_iv_{i+1} \in E(G)\}|=1$.

   Without loss of generality, let $v_1v_2\in E(G)$. Then by the above assumption, $\{v_2, v_3,v_4,v_5\}$ is an
    independent set in $G$.
   Take a 3-coloring $\sigma'$ of $D$ using colors from
   $\{1,2,3\}$ such that
   $\sigma'(v_1)\notin \{ \sigma(t_1), \sigma(t_5)\}$. Let $\sigma''$
   be the 4-coloring of $D$ obtained from $\sigma'$ by changing
   $\sigma'(v_i)$ to $4$ for $i=2,3,4,5$. Then we see that $\sigma$
   and $\sigma''$ form a 4-coloring of $G$, a contradiction.

\medskip

    {\it Subcase} 2.3.  $|\{i: i\in [5] \mbox{ and } v_iv_{i+1} \in E(G)\}|=2$.

     First, suppose for some $i \in [5]$,  $v_iv_{i+1},v_{i+1}v_{i+2}\in E(G)$. Without loss of generality, assume
    $i=5$. If $v_5v_2\notin E(G)$ then the argument for Subcase 2.2 also gives a 4-coloring of $G$, a contradiction.
     So $v_5v_2\in E(G)$. Then $G_1$ has a 4-separation $(G_1',G_1'')$ such
     that $V(G_1'\cap G_1'')=\{v_2,v_3,v_4,v_5\}$, $v_2Cv_5\subseteq G_1'$, and $V(G_1\cap G_2)\subseteq V(G_1'')$.
     Note that $|V(G_1')|\ge 6$. So
     by Lemma~\ref{4cut}, $G_1'$ contains a
      $V(G_1'\cap G_1'')$-good wheel, which is also a $V(G_1\cap G_2)$-good wheel in $G_1$.

     So we may assume without loss of generality that $v_1v_2,
     v_3v_4\in E(G)$. Let $a\in V(v_2Cv_3)\setminus \{v_2,v_3\}$,
    $b\in V(v_4Cv_5)\setminus \{v_4,v_5\}$, and $c\in
    V(v_5Cv_1)\setminus \{v_1,v_5\}$. Note the 5-separation $(H_1,H_2)$ in $G$ with
     $V(H_1)\cap V(H_2)=\{v_i: i\in [5]\}$ and $H_1=D-\{v_1v_2,v_3v_4\}$. Since
     $H_1-V(H_1\cap H_2)$ has three vertices, namely, $a,b$, and $c$,
     it must be a path by Lemma~\ref{5cut3vts}.

     Suppose $ab,ac\in E(G)$.  Then, since $G$ is 4-connected and $(G_1,V(G_1\cap G_2))$ is planar, $cv_2,bv_3\in E(G)$. Now $G$ has a
     separation $(H_1,H_2)$ such that
     $V(H_1\cap H_2)=\{a,t_1,t_2,t_5,v_5\}$ and $H_1-V(H_1\cap H_2)$ is
     the triangle $cv_1v_2c$, contradicting Lemma~\ref{5cuttri}.

    So by symmetry, let $ca,cb\in E(G)$, then $ab\not\in E(G)$. Then, since $G$ is
    4-connected and $(G_1,V(G_1\cap G_2))$ is planar,  $bv_3,av_1\in
    E(G)$. Then $G$ has a separation $(H_1,H_2)$ such that
     $V(H_1\cap H_2)=\{c, t_1,t_2,t_5,v_3\}$ and $H_1-V(H_1\cap H_2)$ is
     the triangle $av_1v_2a$, contradicting Lemma~\ref{5cuttri}.

\medskip

    {\it Subcase} 2.4.  $|\{i: i\in [5] \mbox{ and } v_iv_{i+1} \in E(G)\}|=3$.

    First, suppose for some $i\in [5]$, $v_iv_{i+1}, v_{i+1}v_{i+2}, v_{i+2}v_{i+3}\in E(G)$, where $v_6=v_1$, $v_7=v_2$, and $v_8=v_3$.
    Without loss of generality,
    let $i=3$. Let $a\in V(v_1Cv_2)\setminus \{v_1,v_2\}$ and $b\in V(v_2Cv_3)\setminus \{v_2,v_3\}$. Since $G$
    is 4-connected and $(G_1,V(G_1\cap G_2))$ is planar, $ab\in E(G)$ and   we may assume
    by symmetry between $a$ and $b$ that $bv_4\in E(G)$ and $bv_5\notin E(G)$.
         By assigning $\sigma(t_4)$ to $a$, the color 4 to $v_1,v_2,v_4$,
     and greedily coloring $v_5,v_3,b$ in order, we extend $\sigma$ to a 4-coloring of $G$; a contradiction as
    $G$ is a Haj\'{o}s graph.

   Thus, we may assume, without loss of generality, that $v_4v_5,v_5v_1,v_2v_3\in E(G)$. Then by the above assumption, $v_3v_5\notin E(G)$.  Let $a\in
   V(v_1Cv_2)\setminus \{v_1,v_2\}$ and $b\in V(v_3Cv_4)\setminus
   \{v_3,v_4\}$.  By the planarity of $G_1$ and the symmetry between $a$ and $b$, we may assume that $av_3\notin E(G)$.
    We can now extend $\sigma$ to a 4-coloring of $G$ by assigning the color 4 to $v_1,v_2,v_4$, and greedily coloring $v_5,v_3,b,a$ in order.
    This contradicts the assumption that $G$ is a
      Haj\'{o}s graph.  

\medskip

    {\it Subcase} 2.5. $|\{i: i\in [5] \mbox{ and } v_iv_{i+1} \in E(G)\}|=4$.

   Without loss of generality, we may assume that $|V(v_1Cv_2)|=3$ and
   let $a\in V(v_1Cv_2)\setminus \{v_1,v_2\}$.
   Since $G$ is 4-connected, $a$ is adjacent to at least two of $\{v_3,v_4,v_5\}$.

   We claim that $av_3,av_5\in E(G)$. For, otherwise, by symmetry, assume $av_3\notin E(G)$. Then $av_4,av_5\in E(G)$, and $G$ has a
   5-separation $(H_1,H_2)$ such that $V(H_1\cap H_2)=\{t_1,t_4,t_5,v_2,v_4\}$ and $H_1-V(H_1\cap H_2)$ is the triangle $v_1av_5v_1$, contradicting Lemma~\ref{5cuttri}.

    Moreover, $v_3v_5\in E(G)$. For, otherwise, we can extend $\sigma$ to a proper 4-coloring of $G$ by assigning the color 4 to  $v_1,v_2,v_4$ and
   greedily coloring $v_3,v_5, a$ in order, a contradiction.

   Then $av_4\notin E(G)$ by planarity. Let
   $G':=G-\{a,v_1,v_2\}
   +\{t_1v_3,t_1v_5\}$. Note that if $T$ is a
   $K_5$-subdivision in $G'$ then $(T-\{t_1v_3,t_1v_5\})\cup
   t_1v_1v_5\cup t_1v_2v_3$ (and, hence, $G$) contains a
   $K_5$-subdivision. Hence, $G'$ contains no
   $K_5$-subdivision.
    So let $\sigma'$ be a 4-coloring of $G'$. Now $\sigma'$ can be extended to a 4-coloring of $G$ by assigning $\sigma'(t_1)$ to $a$ and greedily coloring
   $v_1,v_2$ in order. This is a contradiction. 

\medskip

    {\it Subcase} 2.6.  $|\{i: i\in [5] \mbox{ and } v_iv_{i+1} \in E(G)\}|=5$.
    
	By planarity,  $|E(D)\setminus E(C)|\le 2$  and, by symmetry, we may assume $E(D)\setminus E(C)\subseteq \{v_4v_1,v_4v_2\}$.
    Let $G':=G-D+\{t_1t_3,t_3t_4,t_4t_1\}$. Note that if $T$ is a
    $K_5$-subdivision in $G'$ then $(T-\{t_1t_3,t_3t_4,t_4t_1\})\cup
    t_1v_2v_3t_3v_4t_4v_5v_1t_1$ (and, hence, $G$) contains a $K_5$-subdivision. So $G'$ contains no
    $K_5$-subdivision and, thus,
    has a 4-coloring, say $\sigma'$.

     If $\sigma'(t_2)=\sigma'(t_3)$ then by assigning $\sigma'(t_1)$ to
   $v_4$ and greedily coloring $v_5,v_1,v_2,v_3$ in order, we get a 4-coloring of $G$, a contradiction. So $\sigma'(t_2)\ne \sigma'(t_3)$. Similarly,
   $\sigma'(t_4)\ne \sigma'(t_5)$. Then by assigning $\sigma'(t_1), \sigma'(t_3), \sigma'(t_4)$ to $v_4,v_2,v_1$, respectively, and greedily coloring
   $v_3,v_5$ in order, we obtain a 4-coloring of $G$, a contradiction.
 \end{proof}

    We now complete the characterization of  all 5-separations $(G_1,G_2)$ in a Haj\'{o}s graph such that
    $(G_1,V(G_1\cap G_2))$ is planar and $G_1$ contains no $V(G_1\cap G_2)$-good wheel. We say that 
    $(G_1,V(G_1\cap G_2))$ is one of the graphs in  Figure~\ref{obstructions} if 
    $G_1$ is isomorphic to one of the graphs in Firgure~\ref{obstructions} and $V(G_1\cap G_2)$ correspond to the set $S$ there. 

  \begin{lem}\label{5cutobs}
    Let $G$ be a Haj\'{o}s graph and $(G_1,G_2)$ be a 5-separation in $G$ such that  $(G_1,V(G_1\cap G_2))$ is planar and $V(G_1\cap G_2)$ is independent in $G_1$. Then 
   $G_1$ contains a   $V(G_1\cap G_2)$-good wheel, or $(G_1,V(G_1\cap G_2))$ is  one of the graphs in Figure~\ref{obstructions}.
  \end{lem}

\begin{proof}
          First, assume $|V(G_1)\setminus V(G_1\cap G_2)|=1$ and let $u\in V(G_1)\setminus V(G_1\cap G_2)$. Since $G$ is 4-connected,
          $|N_G(u)\cap V(G_1\cap G_2)|\ge 4$. Thus, $(G_1,V(G_1\cap G_2))$ is one of the two 6-vertex graphs in Figure~\ref{obstructions}.
         So we may assume that  $|V(G_1)\setminus V(G_1\cap G_2)|\ge
         2$. We may also assume that

    \begin{itemize}
     \item [(1)] each vertex in $V(G_1\cap G_2)$ must have a neighbor
          in $V(G_1)\setminus V(G_1\cap G_2)$.
    \end{itemize}
       For, otherwise, suppose $x\in V(G_1\cap G_2)$ and $x$ has no neighbor in
          $V(G_1)\setminus V(G_1\cap G_2)$. Then $G$ has a 4-separation $(L_1,L_2)$ such that $V(L_1\cap L_2)=V(G_1\cap G_2)\setminus \{x\}$,
          $L_1\subseteq G_1$ and $|V(L_1)|\ge 6$, and $G_2\subseteq L_2$. Thus,
     by Lemma~\ref{4cut}, $L_1$ contains a
      $V(L_1\cap L_2)$-good wheel, which is also a $V(G_1\cap G_2)$-good wheel in $G_1$. $\Box$

\medskip

     We may  further assume that
         \begin{itemize}
     \item [(2)]     $|V(G_1)\setminus V(G_1\cap G_2)|\ge 4$.
       \end{itemize}
       If  $|V(G_1)\setminus V(G_1\cap G_2)|= 3$ then,
       by
          Lemma~\ref{5cut3vts}, $(G_1,V(G_1\cap G_2))$ is the 8-vertex graph
         in Figure~\ref{obstructions}. Now suppose $|V(G_1)\setminus V(G_1\cap G_2)|=2$ and let
          $V(G_1)\setminus V(G_1\cap G_2)=\{u,v\}$.
          Since $(G_1, V(G_1\cap G_2))$ is planar and $G$ is 4-connected, $uv\in E(G_1)$ and $1\le |N_G(u)\cap N_G(v)|\le 2$.
          If $ |N_G(u)\cap N_(v)|= 2$ then $(G_1,V(G_1\cap G_2))$ is the first 7-vertex graph in Figure~\ref{obstructions}; and if
          $ |N_G(u)\cap N_G(v)|= 1$ then $(G_1,V(G_1\cap G_2))$ is the
          second 7-vertex graph in Figure~\ref{obstructions}.  $\Box$

       \medskip

     By  Lemma~\ref{5cutdeg2}, we may assume that
         there exists $x\in V(G_1\cap G_2)$ such that $x$ has exactly one neighbor in $V(G_1)\setminus V(G_1\cap G_2)$, say $y$.
          Let $H_1=G_1-x-\{yz: z\in V(G_1\cap G_2)\}$ and $H_2=G-(V(H_1)\setminus V(H_2))$. Then $(H_1,H_2)$ is a separation in $G$. 
            Let $V(H_1\cap H_2)=
                \{p,q,r,s,t\}$ and assume that $H_1$ is drawn in a closed disc in the plane with no edge crossings
                such that $p,q,r,s,t$ occur on the boundary of the disc in
                 clockwise order. Note that $y\in \{p,q,r,s,t\}$.
     We may  assume that

     \begin{itemize}
     \item [(3)]
          if $|V(G_1)\setminus V(G_1\cap G_2)|=4$ then $G_1$ is the 9-vertex graph in Figure~\ref{obstructions}.
      \end{itemize}
         Suppose $|V(G_1)\setminus V(G_1\cap G_2)|=4$.  Then
         $|V(H_1)\setminus V(H_1\cap H_2)|=3$ and,
               by Lemma~\ref{5cut3vts}, $(H_1, V(H_1\cap H_2))$ is the 8-vertex graph in Figure~\ref{obstructions}.
                Without loss of generality, let $d_{H_1}(p)=3$,
                $d_{H_1}(r)=d_{H_1}(s)=2$, and
                $d_{H_1}(q)=d_{H_1}(t)=1$. Then $d_G(p)\ge 5$ (by Lemma~\ref{5cut3vts}). Let $V(H_1)\setminus
                V(H_1\cap H_2)=\{u,v,w\}$ such that $quvwt$ is a path.

                First, consider the case $y\in \{q,t\}$. By symmetry, assume $y=q$. Since $G$ is 4-connected, $d_{G_1}(q)\ge 4$; so
                $qp,qr\in E(G)$ as $(G_1,V(G_1\cap G_2))$ is planar. Then $\{p,q,r,u,v\}$ induces a  $V(G_1\cap G_2)$-good wheel in $G_1$.

                Now assume $y\in \{r,s\}$ and, by symmetry, let $y=s$. We may assume  $rs\notin E(G)$, as otherwise, $\{p,r,s,u,v,w\}$ induces
                a $V(G_1\cap G_2)$-good wheel in $G_1$. 
                Hence $G$ has a separation $(L_1,L_2)$ such that
                $V(L_1\cap L_2)=\{p,t,v,x\}$, $G_2\subseteq L_2$, and  $|V(L_1)|=6$; which contradicts Lemma~\ref{4cut}.

                Thus, $y=p$. If $pq,pt\in E(G)$ then $G_1$ is the 9-vertex graph in Figure~\ref{obstructions}. So assume by symmetry that
                $pq\notin E(G)$. Then $G$ has a separation $(L_1,L_2)$ such that  $V(L_1\cap L_2)=\{r,s,t,u,x\}$ and $L_1-V(L_1\cap L_2)$ is the
                triangle $pvwp$, contradicting Lemma~\ref{5cuttri}.  $\Box$

 \medskip
               
		By (3), we may assume   $|V(G_1)\setminus V(G_1\cap G_2)|\ge 5$ and,
               subject to this, we may choose $(G_1,G_2)$ so that $G_1$ is minimal.
                   Then by the choice of $(G_1,G_2)$,
                   $|V(H_1)\setminus V(H_1\cap H_2)|=4$. So  by (3),
                $(H_1, V(H_1\cap H_2))$ is the 9-vertex graph in Figure~\ref{obstructions}.
                Without loss of generality,  let $d_{H_1}(p)=1$, and let $V(H_1)\setminus
                V(H_1\cap H_2)=\{u,v,w,z\}$ such that $d_{H_1}(z)=6$, $d_{H_1}(v)=5$, and  $quvwt$ is a path in $H_1$.

                If $y=p$ then, since $G$ is 4-connected and $(G_1,V(G_1\cap G_2))$ is planar, $pq,pt\in E(G)$; so $G[\{p,q,t,u,v,w,z\}]$ is a
                $V(G_1\cap G_2)$-good wheel in $G_1$. 
                 If $y\in \{r,s\}$ then by symmetry assume $y=r$; now  $rq\in E(G)$ or $rs\in E(G)$ as $G$ is 4-connected and $(G_1, V(G_1\cap G_2))$ is planar;
                 so  $G[\{q,r,u,v,z\}]$ or $G[\{r,s,u,v,w,z\}]$  is a
                $V(G_1\cap G_2)$-good wheel in $G_1$. 

                So we my assume $y\in \{q,t\}$ and, by symmetry, let $y=q$.
                Since $G$ is 4-connected and $(G_1,V(G_1\cap G_2))$ is planar, $qp\in E(G)$  or $qr\in E(G)$.
                If $qr\in E(G)$ then $G[\{q,r,u,v,z\}]$ is a $V(G_1\cap G_2)$-good in $G_1$. 
                  Hence, we may assume $qr\notin E(G)$ and $qp\in E(G)$.

              Let  $G':=G-\{q,u,v,w,z\}+\{rp,rt,pt,px,ts\}$.  If $G'$ has no $K_5$-subdivision then
                $G'$ has a 4-coloring, say  $\sigma$. By assigning $\sigma(r),\sigma(p),\sigma(t)$ to
                $z,u,v$, respectively, and greedily coloring $q,w$ in order, we obtain a 4-coloring of $G$, a contradiction.

              Hence, $G'$ has a $K_5$-subdivision, say $T$, and let
                 $T'=T-\{ rp,rt,pt,px,ts\}$.
                 
\medskip

            {\it   Case} 1.  $\{tp,tr,pr\}\subseteq E(T)$.

         Then $t,p,r$ must be branch vertices of $T$.
                If $ts,px\in E(T)$ then $T'$ and $wt,ws,wvr,wz,zp$, $zqx,zur$ from a $K_5$-subdivision in $G$
                (by replacing branch vertices $t,p$ with $w,z$, respectively), a contradiction.
                If $ts\in E(T)$ and $px\notin E(T)$ then  $T'\cup
                wt\cup ws\cup wvr\cup wzp\cup pqur$
                is a $K_5$-subdivision in $G$, a contradiction. If $ts\notin E(T)$ and $px\in E(T)$ then
                $T'\cup zp\cup zt\cup zqx\cup zur\cup twvr$ is a $K_5$-subdivision in $G$, a contradiction.
                If $ts,px\notin E(T)$ then $T'\cup tzpqurvwt$ is a $K_5$-subdivision in $G$, a contradiction.
                
\medskip

           {\it Case} 2.    $tp\notin E(T)$ for any choice of $T$. 

                Then  $p$ or $t$
                is not a branch vertex in $T$ as, otherwise, by replacing the path in $T$ between $p$ and $t$ with $pt$ we would obtain a
                $K_5$-subdivision in $G'$ containing $tp$. 

                             If $p$ is not a branch vertex of $T'$ then $T'\cup tws\cup
             tzvr\cup ruqx$ (when $rp,px\in E(T)$), or $T'\cup tws\cup
             tzvr\cup ruqp$ (when $px\notin E(T)$), or $T'\cup tws\cup
             tzvr\cup pqx$ (when $rp\notin E(T)$), contains a $K_5$-subdivision. So
             $G$ contains a $K_5$-subdivision, a contradiction.

             If $t$ is
             not a branch vertex of $T'$ then $T'\cup pqx\cup pzur\cup rvs$ (when $tr,ts\in E(T)$),
             or $T'\cup pqx\cup pzur\cup rvwt$ (when $ts\notin E(T)$),
             $T'\cup pqx\cup pzur\cup tws$ (when $tr\notin E(T)$), contains a $K_5$-subdivision.
              Hence, $G$ contains
             a $K_5$-subdivision, a contradiction.

\medskip

         {\it Case} 3.    $\{tp,ts,tr\}\subseteq E(T)$ or
           $\{pt,pr,px\}\subseteq E(T)$.

   If $\{tp,ts,tr\}\subseteq E(T)$ then we may assume $pr\notin E(T)$ by Case 1; thus $T'\cup wt\cup ws\cup
           wzp\cup wvr\cup pqx$ contains a $K_5$-subdivision (with branch vertex $t$ replaced by $w$), a contradiction. If
           $\{pt,pr,px\}\subseteq E(T)$ then we may assume $tr\notin E(T)$ by Case 1; thus $T'\cup zp\cup zt\cup
           zqx\cup zur\cup tws$ contains a $K_5$-subdivision (with branch vertex $p$ replaced by $z)$, a contradiction.

\medskip

            {\it Case} 4. $tp\in E(T)$, $|\{ts,tr\}\cap E(T)|\le 1$, and
            $|\{pr,px\}\cap E(T)|\le 1$.

      Then $\{rp,rt,pt,px,ts\}\cap E(T)$ is contained in one of the
      following paths: $stpx$, $stpr$, and $rtpx$. Now $T'\cup
      swtzpqx$, or $T'\cup swtzpqur$, or $T'\cup rvwtzpqx$ is a
      $K_5$-subdivision in $G$, a contradiction.                  
\end{proof}

\section{Conclusion}

It is clear that Theorem~\ref{main} follows from Lemmas~\ref{4cut} and \ref{5cutobs}.
Our motivation to prove Theorem~\ref{main} is to use wheels to construct subdivisions of $K_5$: Let $W$ be a wheel with center $w$ and
let $w_1,w_2,w_3,w_4$ be distinct neighbors of $w$ on the cycle $W-w$ in cyclic order. If $P_1,P_2$ are disjoint paths from $w_1,w_2$ to $w_3,w_4$, respectively, and internally disjoint from W, then $W\cup P_1\cup P_2$ form a subdivision of $K_5$.

We know from results in \cite{Yu06, KSC4} that every Haj\'{o}s graph has a 4-separation. Now suppose $(G_1,G_2)$ is a 4-separation
in $G$ with $V(G_i)\setminus V(G_{3-i})\ne \emptyset$ for $i\in [2]$. We have shown that if $(G_1, V(G_1\cap G_2))$ is planar and $|V(G_1)|\ge 6$ then $G_1$ contains a $V(G_1\cap G_2)$-good wheel, say $W$.
In a subsequent paper, we will show that we can extend $W$ to $V(G_1\cap G_2)$ by four disjoint paths from four distinct neighbors of $w$ on $W-w$ to $V(G_1\cap G_2)$. If $G_2$ has the right disjoint paths between the vertices of $V(G_1\cap G_2)$ then we can form a subdivision
of $K_5$ in $G$. If not then by Seymour's characterization of 2-linked graphs, we will see that $(G_2,V(G_1\cap G_2)$ is planar and, hence,
$G$ would be planar.

As a consequence, no Haj\'{o}s graph admits a 4-separation $(G_1,G_2)$ such that $|V(G_1)|\ge 6$ and $(G_1, V(G_1\cap G_2))$
is planar. This is an important step in modifying  the recent proof of the Kelmans-Seymour conjecture  in  \cite{KSC1, KSC2, KSC3,KSC4} to
make progress on the Haj\'{o}s conjecture; in particular, for the class of graphs containing
$K_4^-$ as a subgraph, where $K_4^-$ is the graph obtained from $K_4$ by removing an edge. The arguments in  \cite{KSC2, KSC3}
heavily depend on the assumption of 5-connectedness, and we wish to  replace such arguments with coloring arguments. For this to work, we need to first deal with 4-separations with a planar side, similar to the result in \cite{MY10} on 5-cuts with a planar side.

\end{document}